\documentclass{amsart}%
\usepackage{amsfonts}
\usepackage{amsmath}
\usepackage{amssymb}
\usepackage{epsfig}
\usepackage{graphicx}
\usepackage{version}
\usepackage{caption}%
\newcommand{\R}{{\mathbb R}}

\theoremstyle{plain}
\newtheorem{theorem}{Theorem}[section]

\newtheorem{proposition}{Proposition}[section]
\theoremstyle{definition}
\newtheorem{definition}{Definition}[section]
\newtheorem{example}{Example}[section]
\newtheorem{remark}{Remark}[section]
\newtheorem*{notation*}{Notation}
\numberwithin{equation}{section}
\excludeversion{comment1}
\begin{document}
\title{Discontinuous Fractal Functions and Fractal Histopolation}
\author[M. F. Barnsley]{Michael F. Barnsley}
\address{Australian National University}
\author[P. Viswanathan]{P. Viswanathan}
\address{Australian National University}

\begin{abstract}
Fractal functions that produce smooth and non-smooth approximants constitute  an advancement to classical nonrecursive methods of approximation.
In both classical and fractal approximation methods emphasis is given for investigation of continuous approximants whereas much real data demand discontinuous models. This article intends to point out that many of the results  on fractal functions in the traditional setting can be immediately  extended to the discontinuous case. Another topic is the study of  area matching properties of integrable fractal functions in order to introduce the concept of fractal histopolation.

\end{abstract}
\maketitle

\section{Introductory Remarks} \label{ffhistosec1}
Approximation theory, which primarily focuses on the approximation of real-valued continuous
functions by some simpler class of functions, covers a great deal of mathematical territory.
Beyond polynomials, trigonometric functions and splines, multitudes of approximation tools were developed which were often directed to a specific  application domain. In the historical development of ``classical" approximation theory, all efforts were directed towards study of  smooth approximants. On the other hand,  many experimental and real world signals rarely show sensation of smoothness in their traces and hence demand  rough functions for an effective representation.

\par Following Benoit Mandelbrot's fractal vision  of the universe,
the notion of fractal function was introduced  in reference \cite{Barnsley1}. Subsequently, it was observed that fractal functions can be used for smooth approximation as well
\cite{BH}, thereby supplementing  various traditional approximation techniques. For more than a quarter century of its first pronouncement, theory of
fractal functions has been extensively researched  and has evolved beyond its mathematical framework; see, for instance, \cite{BEM, BD,CV,GH,Mas,N1,VC2,WY}. Specific applications of fractal interpolation function include medicine, physics, and economics with, for instance, in the study  of tumor perfusion \cite{CZ}, electroencephalograms \cite{SN}, turbulence \cite{BFP}, speech signals \cite{VDL}, signal processing \cite{PH}, stock market \cite{YX}  etc.
\par
The fractal functions were originally introduced as continuous functions interpolating a prescribed set of data. Thus, the approximation methods, both  fractal functions and their precedents, are based on continuity, and methods for discontinuous approximation
are limited. However, many real data requires to be modeled by discontinuous functions. For instance, discontinuous functions arise as solutions of partial differential equations describing different types of  systems from classical physics. Since machines have only finite precision, any functions as represented by machine will be discontinuous and so is anything which involves discretization. Mandelbrot's view of the price movement in competitive markets is  one of the profound sources of discontinuity. In words of Mandelbrot \cite{Mandel}: ``...But prices on competitive markets need not be continuous,  and they are conspicuously discontinuous. The only reason for assuming continuity is that many sciences tend, knowingly or not, to copy the procedures that prove successful in Newtonian physics...." Hence it appears that Mandelbrot envisaged not only a world of irregularity, but ultimate discontinuity.
\par
 Deriving principle influence from these facts,  the current  article intends to investigate discontinuous fractal functions and their elementary properties. These fractal functions may not interpolate the data set, thus breaking the inherent  continuity and interpolatory nature of the fractal functions  in the traditional setting. Discontinuous fractal functions reported recently  in \cite{NAVdisFIF} provides one of the impetus for the current study. However, our main aim is to show that the tools developed in 1986 paper \cite{Barnsley1} are well-suited to deal with discontinuous fractal functions as well. Further,  we convey  that the algebraic structure of equations governing fractal functions and their moment integrals are similar to that of usual continuous fractal interpolation function, but provides extra degrees of freedom. To this end, first we prove that the set of points of discontinuity of these fractal functions has Lebesgue measure zero; a fact that ensures Riemann integrability. As a consequence,  functional equations for various integral transforms of these discontinuous fractal functions can be easily derived, thus finding potential applications in various fields in science and engineering. For a  restricted  class of discontinuous fractal functions and its connection with Weierstrass type functional equations, the reader is invited to refer \cite{BBVV}.

\par
Interpolation ``represents" a function by preserving function values at prescribed knot points. A closely related but different approximation method is Histopolation which  preserves integrals
of the function over the intervals of histopolation. Given a mesh $\Delta:= \{x_0,x_1, \dots, x_N\}$ with strictly increasing knots and a histogram $F=\{f_1,f_2,\dots, f_N\}$, that is $f_i$ is the frequency for the interval $[x_{i-1},x_i]$ with mesh spacing $h_i$, $i=1,2,\dots, N$, histopolation seeks to find a function $f$ that satisfies the ``area" matching condition:
$$\int_{x_{i-1}}^{x_i} f(x)~\mathrm{d}x =f_i h_i.$$
\par There are  several models that lead to histopolation problem. For instance, $F$ may be obtained from a finite sample with the observed frequency $f_i$ in the class interval $[x_{i-1}, x_i)$ for $i=1,2,\dots, N$ and area matching function $f$ may be taken as approximation of the unknown density function of the underlying random variable. The following fundamental  problem occurring in one dimensional
motion of a material point is another example where the problem of histopolation emerges quite naturally. Suppose we have to model velocity $f(x)$ of the material point wherein  $g(x)$ is the position of the point at time $x$, which is known at specified points $x_0<x_1<\dots<x_N$. Then we have
$$ g(x) = g(x_0) + \int_{x_0}^x f(t) ~\mathrm{d}t, \quad x_0 \le t\le x.$$
Since
$$ \int_{x_{i-1}}^{x_i} f(x) ~\mathrm{d}x= g(x_i) -g(x_{i-1}), \quad i=1,2,\dots,N $$
a representative for $f$ can be obtained by solving the histopolation problem with values in the histogram $F$ given by
$$f_i= \frac{g(x_i) -g(x_{i-1})}{h_i}.$$
In contrast to the vast literature available on interpolation, the researches on histopolation is limited. In case of smooth histopolant $f$, the problem of histopolation can be easily transformed in to a problem of interpolation as follows. If we construct an interpolant $g \in \mathcal{C}^1[x_0,x_N]$ with interpolation conditions $g(x_i) = \sum_{j=1}^i h_jf_j + g(x_0)$, $i=1,2,\dots, N$, where $g(x_0)$ is arbitrary, then $f=g'$ solves the histopolation problem. This may be one of the reasons for the obscurity of histopolation as a separate problem. However, the situation is different in the case of  fractal functions as they are not differentiable in general. At the same time, as in the case of interpolation, constructing rough histopolants is of practical relevance. Owing to these reasons, fractal histopolation deserves special attention and it is to this that the last section of the paper focus on. \\
Overall, the current study may be viewed as an attempt to revitalize fractal functions and their applications and to initiate a study on fractal histopolation.
\section{Continuous  Fractal Interpolation Function: Revisited}\label{ffhistosec2}
To make the article fairly self-contained, we shall briefly evoke the notion of  Fractal Interpolation Function  and associated  concepts in this section. To prepare the setting, first we need the following definition.
\begin{definition}
Let $(X,d)$ be a complete metric space. Let $m >1$ be a positive integer and let $w_i: X \to X$ for $i=1,2,\dots, m$ be continuous mappings. Then the
collection $\{X; w_1,w_2,\dots, w_m\}$ is called
an Iterated Function System, IFS for short.
\end{definition}
\noindent For a given  IFS $\mathcal{F}=\{X; w_1,w_2,\dots, w_m\}$ one can associate a set-valued map, which is termed collage map, as follows.
Let $\mathcal{H}(X)$ denote the collection of all non-empty compact subsets of $X$ endowed with the Hausdorff metric
\[
h(A,B)= \max \Big\{\max_{a \in A} \min_{b \in B} d(a,b), \max_{b \in B} \min_{a \in A} d(b,a)\Big\} ~ \forall A, B \in \mathcal{H}(X).
\]
Define
$W: \mathcal{H}(X) \to \mathcal{H}(X)$ by
\[ W(B) = \cup_{i=1}^m w_i (B),\]
where $w_i(B)=\{w_i(b): b \in B\}$.
\begin{definition}
A nonempty compact subset $A$ of $X$ is called an attractor of an IFS $\mathcal{F}=\{X; w_1,w_2,\dots, w_m\}$ if
\begin{enumerate}
\item $A$ is a fixed point of $W$, that is $W(A)=A$

\item there exists an open set $U \subseteq X$ such that $A \subset U$ and $\lim_{k \to \infty} h\big( W^k (B),A)=0$ for all $B \in \mathcal{H}(U)$, where
$W^k$ is the $k$-fold autocomposition of $W$.
\end{enumerate}
The largest open set $U$  for which (2) holds is called the basin of attraction for
the attractor $A$ of the IFS $\mathcal{F}$. The attractor $A$ is also referred to as a fractal or self-referential set owing to the fact that $A$ is a union of transformed copies of itself.
\end{definition}
\noindent If the IFS $\mathcal{F}$ is contractive (hyperbolic), that is each map $w_i$ is a contraction map, then the existence of a unique  attractor is ensured by the Banach fixed point theorem and in this case the basin of attraction is $X$.\\
In what follows, the question of how to obtain continuous functions whose graphs are fractals in the above
sense is readdressed.\\
Let $\{(x_i, y_i) : i = 0, 1,\dots,N\}$ denote the cartesian coordinates of a finite set of
points with increasing abscissae in the Euclidean plane $\mathbb{R}^2$. Let $I$ denote the closed bounded interval $[x_0,x_N]$ and $I_i=[x_{i-1},x_i]$
for $i=1,2,\dots, N$. Suppose $L_i : I \to I_i$ be contraction homeomorphism such that
\begin{equation}\label{DFIF1}\left\{
\begin{split}
&L_i(x_0)=x_{i-1}, \quad L_i(x_N)=x_i, \\
&|L_i(x) -L_i(x^*)| \le l_i |x-x^*| \quad \forall~ x, x^* \in I, ~ l_i \in (0,1).
\end{split}\right.
\end{equation}
Note that $\{I; L_i, i=1,2,\dots, N\}$ is a hyperbolic IFS with unique attractor $I$. Further, assume that $F_i: I \times \mathbb{R} \to \mathbb{R}$ is continuous map satisfying
\begin{equation}\label{DFIF2}\left\{
\begin{split}
& F_i(x_0,y_0)=y_{i-1}, \quad F_i(x_N,y_N)= y_i,\\
& \big|F_i(x,y)-F_i(x,y^*)\big| \le s_i |y-y^*| \quad  \forall~ y,y^* \in \mathbb{R}, ~ s_i \in (0,1).
\end{split}\right.
\end{equation}
Now define functions $W_i : I \times \mathbb{R} \to I \times \mathbb{R}$ by $$W_i(x,y) = \big(L_i(x), F_i(x,y) \big).$$ The following is a fundamental theorem  that leads to the definition of Fractal Interpolation Function.
\begin{theorem}(\cite{Barnsley1})
The IFS $\{I \times \mathbb{R}; W_i, i=1,2,\dots, N\}$ has a unique attractor $G(f)$ which is  the graph of a continuous function $f: I \to \mathbb{R}$ satisfying $f(x_i)=y_i$ for $i=0,1,\dots, N$.
\end{theorem}
\begin{definition}\label{DFIF3}
The function $f$ that made its debut in the foregoing theorem is termed Fractal Interpolation Function (FIF) corresponding to the data $\{(x_i,y_i): i=0,1,\dots, N\}$.
\end{definition}
\noindent To obtain a functional equation for $f$, one may proceed as follows.\\
Denote by $\mathcal{C}(I)$ the space of all continuous functions defined on $I$ endowed with the  Chebyshev norm
 $$\|g\|_\infty:= \max \{|g(x)|: x \in I\}$$
and consider the closed (metric) subspace
$$\mathcal{C}_{y_0,y_N}(I):= \{g \in
\mathcal{C}(I): g(x_0)=y_0, g(x_N)=y_N\}.$$
Define an operator, which is a form of Read-Bajraktarivi\'{c} operator, $T: \mathcal{C}_{y_0,y_N}(I) \rightarrow \mathcal{C}_{y_0,y_N}(I)$
$$(Tg)(x)=F_i\big(L_i^{-1}(x), g \circ L_i^{-1}(x)\big), ~ x \in
I_i, ~i \in \{1,2,\dots, N\}.$$
\begin{theorem} (\cite{Barnsley1})
The operator $T: \mathcal{C}_{y_0,y_N}(I) \rightarrow \mathcal{C}_{y_0,y_N}(I)$ is a contraction with a contractivity factor $s:=\max\{s_i: i=1,2,\dots, N\}$ and the fixed point of $T$ is the FIF $f$ corresponding to the data $\{(x_i,y_i):i=0,1,\dots,N\}$. Consequently, $f$ satisfies the functional equation
$$f\big(L_i(x) \big) = F_i\big(x,f(x) \big), ~x \in I, ~i=1,2,\dots, N.$$
\end{theorem}
\noindent The most widely studied FIFs in theory and applications are defined by IFS with maps
\begin{equation}\label{DFIF4}\left\{
\begin{split}
L_i (x) =& ~ a_i x + b_i\\
F_i(x,y)=&~\alpha_i y + q_i(x),
\end{split}\right.
\end{equation}
where $|\alpha_i|<1$ and $q_i: I \to \mathbb{R}$ is continuous function satisfying
$$q_i(x_0)= y_{i-1} -\alpha_i y_0, \quad q_i(x_N)=y_i -\alpha_i y_N.$$
The subinterval end point restraints yield
$$a_i= \frac{x_i-x_{i-1}}{x_N-x_0}, \quad b_i= \frac{x_Nx_{i-1}-x_0x_i}{x_N-x_0}.$$
The parameter $\alpha_i$ is called vertical scaling factor of the map $W_i$ and the vector $\alpha=(\alpha_1,\alpha_2,\dots, \alpha_N) \in (-1,1)^N$ is refereed to as scale vector of the IFS. If $q_i$, $i=1,2,\dots, N$ are affine maps, then the FIF is termed affine FIF. In this case, $q_i (x) = q_{i0} x + q_{i1}$, where
$$ q_{i0}=\frac{y_i-y_{i-1}}{x_N-x_0} -\alpha_i \frac{y_N-y_0}{x_N-x_0}, \quad q_{i1}=\frac{x_N y_{i-1}-x_0y_i}{x_N-x_0}-\alpha_i \frac{x_Ny_0-x_0y_N}{x_N-x_0}. $$
The following special choice of $q_i$ in (\ref{DFIF4}) is of special interest.
\begin{equation}\label{DFIF5}
q_i(x)=h\circ L_i(x) - \alpha_i b(x),
\end{equation}
where the height function $h$ is a continuous  interpolant to the data and  base function $b$  is a continuous function that passes though the extreme points $(x_0,y_0)$ and $(x_N,y_N)$. Note that in case of affine FIF, $h$ is piecewise linear function with vertices at the data points $\{(x_i,y_i): i=0,1,2,\dots, N\}$ and $b$ is a line joining the extremities of the interpolation interval. In contrast to traditional nonrecursive interpolants, the FIF $f$ corresponding to (\ref{DFIF4}) is, in general,  nondifferentiable. For instance, we have
\begin{theorem}(\cite{LS})
Let $\{(x_0,y_0), (x_1,y_1), \dots, (x_N,y_N)\}$ be an  equally spaced  data set in $I=[x_0,x_N]=[0,1]$. Consider the IFS defined through the maps (\ref{DFIF4})-(\ref{DFIF5}), where $h \in \mathcal{C}^1[0,1]$, $|\alpha_i| \ge \frac{1}{N}$ for $i=1,2,\dots, N$, and $h'(x)$ does not agree with $y_N-y_0$ in a nonempty open subinterval of $I$. Then the set of points at which the corresponding FIF defined by $$f \big(L_i(x)\big) = h \big(L_i(x)\big) + \alpha_i (f-b)(x)$$ is not differentiable is dense in $I$.
\end{theorem}
\noindent However, by proper choices of elements in the IFS, a fractal function $f \in \mathcal{C}^k(I)$, $k \in \mathbb{N} \cup \{0\}$ can be constructed and this is the content of the following theorem.
\begin{theorem} (\cite{BH})
Let $\{(x_i,y_i):i=0,1,2,\dots, N\}$ be a prescribed set of interpolation data with increasing abscissae.
Consider the IFS $\{I \times \mathbb{R}; W_i, i=1,2,\dots, N\}$, where $W_i(x,y)=\big(L_i(x), F_i(x,y)\big)$, $L_i(x) = a_i x + b_i$, and $F_i(x,y)=\alpha_i y
+ q_i(x)$. Suppose that for some
integer $k\geq0$, $|\alpha_i|< a_i^k$, and $q_i\in
\mathcal{C}^k(I)$ for $i=1,2,\dots,N$. Let
\begin{eqnarray*}\label{m8}
F_{i,r}(x,y)=\frac{\alpha_i y+ q_i^{(r)}(x)}{a_i^r};~ y_{1,r} =
\frac {q_1^{(r)}(x_1)}{a_1^r-\alpha_1},
y_{N,r}=\frac{q_{N}^{(r)}(x_N)}{a_{N}^r-\alpha_{N}},~~r =
1,2,\dots,k.
\end{eqnarray*}
If $ F_{i-1,r}(x_N, y_{N,k})=F_{i,r}(x_0,y_{0,k})$  for $\ i =2,3,
\dots,N$ and  $\ r =1,2,\dots,k$, then the IFS
$\{I \times \mathbb{R};W_i:i=1,2,\dots, N\}$ determines a FIF $f\in
\mathcal{C}^k(I)$, and $f^{(r)}$ is the fractal function determined by the IFS
$\{I \times \mathbb{R};\big(L_i(x),F_{i,r}(x,y)\big),i=1,2,\dots,N\}$\: for $ r =
1,2,\dots,k$.
\end{theorem}


\section{Fractal Functions of More General Nature} \label{ffhistosec1}
In this section we note that by  simple modifications in the construction of continuous fractal interpolation function revisited in the previous section, we can break  continuity and/or interpolatory property of the fractal function, providing more flexibility. Although the actual fractal function
appearing in each case discussed in the sequel may be different with different properties, we shall steadfastly employ the same notation $f$ at each appearance. \\\\
\textbf{Case 1: Continuous (but not interpolatory) fractal function}\\\\
For the data $\{(x_i,y_i): i=0,1,2,\dots, N\}$, consider the IFS defined by the maps
\begin{equation}\label{DFIF7}\left\{
\begin{split}
L_i(x) =& ~ a_i x + b_i\\
F_i(x,y)=&~\alpha_i y + q_i(x),
\end{split}\right.
\end{equation}
where the continuous functions $q_i: I \to \mathbb{R}$ are chosen  such that
\begin{equation}\label{DFIF8}\left\{
\begin{split}
F_1(x_0,y_0) =&~y_0,\\
F_N(x_N,y_N)=&~y_N,\\
F_i(x_N,y_N)=F_{i+1}(x_0,y_0)=&~ \tilde{y_i},~ \tilde{y_i}\in [ y_i-\epsilon, y_i+ \epsilon], ~i=1,2,\dots, N-1.
\end{split}\right.
\end{equation}
Here $\epsilon$ may be interpreted as error tolerance in measurement or noise.
\begin{theorem}\label{thm1}
The fractal function $f$ corresponding to the IFS defined through (\ref{DFIF7})-(\ref{DFIF8}) is continuous and satisfies $|f(x_i)-y_i| \le \epsilon$ for all $i=0,1,\dots, N$.
\end{theorem}
\begin{proof}
Consider the closed metric subspace $\mathcal{C}_{y_0,y_N} (I):=\{g \in
\mathcal{C}(I): g(x_0)=y_0, g(x_N)=y_N\}.$ of $\mathcal{C}(I)$. Define $T: \mathcal{C}_{y_0,y_N} (I) \to \mathcal{C}_{y_0,y_N} (I)$ by
$$ (Tg) (x) = F_i \big(L_i^{-1}(x), g \circ L_i^{-1}(x)\big)=\alpha_i g \big( L_i^{-1}(x)\big) + q_i \big(L_i^{-1}(x)\big), ~x \in I_i=[x_{i-1}, x_i], ~i=1,2,\dots,N.$$ It follows at once that $Tg$ is continuous on $I_i=[x_{i-1}, x_i]$ for each $i=1,2,\dots, N$. Bearing in mind that $L_{i}^{-1}(x_i)=x_N$ and $L_{i+1}^{-1}(x_i)=x_0$, for $i=1,2,\dots, N-1$ we have
\begin{equation*}
\begin{split}
(Tg)(x_i^-)=&~ F_i\big(L_i^{-1}(x_i), g\big(L_i^{-1}(x_i)\big)=F_i(x_N, g(x_N))=F_i(x_N, y_N)= \tilde{y_i}\\
(Tg)(x_i^+)=&~ F_{i+1}\big(L_{i+1}^{-1}(x_i), g\big(L_{i+1}^{-1}(x_i)\big)= F_{i+1} (x_0, g(x_0))=F_{i+1}(x_0,y_0)=\tilde{y_i}.
\end{split}
\end{equation*}
Thus, $Tg$ is continuous at each of the internal knots, and consequently on $I=[x_0,x_N]$. On similar lines,
\begin{equation*}
\begin{split}
(Tg)(x_0)=&~ F_1\big(L_1^{-1}(x_0), g(L_1^{-1}(x_0))\big)=F_1(x_0, g(x_0))=F_1(x_0, y_0)= y_0\\
(Tg)(x_N)=&~ F_N\big(L_{N}^{-1}(x_N), g(L_{N}^{-1}(x_N))\big)= F_N (x_N, g(x_N))=F_N(x_N,y_N)=y_N,
\end{split}
\end{equation*}
demonstrating that $Tg$ is a well-defined map on $\mathcal{C}_{y_0,y_N}(I)$. Furthermore, for $x,y \in I_i$
\begin{equation*}
\begin{split}
|(Tg)(x)-(Th)(x)| =& ~ |\alpha_i| \big|g\big(L_i^{-1}(x)\big)-h\big(L_i^{-1}(x)\big) \big|\\
\le &~ |\alpha_i| \|g-h\|_\infty
\end{split}
\end{equation*}
Denoting $|\alpha|_\infty = \max \{|\alpha_i|: i=1,2,\dots, N\}$, the previous inequality stipulates
$$\|Tg-Th\|_\infty \le |\alpha|_\infty \|g-h\|_\infty,$$
 proving contractivity of $T$ in Chebyshev norm. Hence by the Banach fixed point theorem, $T$ has a unique fixed point $f$. Note that $f(x_0)=y_0$, $f(x_N)=y_N$, and
$$|f(x_i)-y_i| = |(Tf)(x_i)-y_i|= \big|F_i\big(L_i^{-1}(x_i), f(L_i^{-1}(x_i))\big)-y_i\big|=|F_i(x_N,y_N)-y_i| \le \epsilon,$$
completing the proof.
\end{proof}
\begin{remark}
Treating $y_0$ and $y_N$ as parameters and replacing third equation in (\ref{DFIF8}) with the condition $F_i(x_N,y_N)=F_{i+1}(x_0,y_0)$, we obtain a continuous fractal function not attached to any data set.
\end{remark}
\textbf{Case 2: Interpolatory (but not continuous) fractal function}\\\\
For a prescribed data set $\{(x_i,y_i):i=0,1,\dots,N\}$, set $I=[x_0,x_N]$, $I_i=[x_{i-1},x_i)$ for $i=1,2,\dots, N-1$ and $I_N=[x_{N-1},x_N]$. Let
$L_i: [x_0,x_N) \to [x_{i-1},x_i)$ be affine maps such that $L_i(x_0)=x_{i-1}$ and $L_i(x_N^-)=x_i$ for $i=1,2,\dots, N-1$ and $L_{N}: I \to I_N$ be affine map satisfying $L_N(x_0)=x_{N-1}$ and $L_N(x_N)=x_N$. Observe that in contrast to the case of continuous fractal function, here we deal with half-open subintervals with obvious modification for the last subinterval  so that each $x_i$ belongs to a subinterval univocally. Let $q_i : I \to \mathbb{R}$ be bounded function  so that  the bivariate map
$F_i: I \times \mathbb{R} \to \mathbb{R}$ defined by $F_i(x,y)= \alpha_i y + q_i(x)$ satisfy
$$F_i (x_0,y_0)= y_{i-1}, ~i=1,2,\dots, N; \quad F_N(x_N,y_N)=y_N.$$
\begin{theorem}
Consider the IFS determined by  the maps $L_i$ and $F_i$ defined in the previous paragraph. The corresponding fractal function $f$ is bounded and satisfies $f(x_i)=y_i$ for $i=0,1,\dots, N$.
\end{theorem}
\begin{proof}
Note that the set $\mathcal{B}(I)$ of all real valued  bounded functions defined on $I$ endowed with the supremum norm is a Banach space. Consider the closed metric subspace
$$\mathcal{B}_{y_0,y_N} (I) = \{g \in \mathcal{B}(I): g (x_0)= y_0, ~g(x_N)=y_N\}$$
of $\mathcal{B}(I)$ and define $T: \mathcal{B}_{y_0,y_N}(I) \to \mathcal{B}_{y_0,y_N}(I)$ by
$$(Tg)(x) = F_i \big(L_i^{-1}(x), g \big(L_I^{-1}(x)\big) \big) = \alpha_i g \big( L_i^{-1}(x)\big) + q_i \big(L_i^{-1}(x)\big), ~x \in I_i,~ i=1,2,\dots,N.$$
It is plain to see that $Tg$ is a bounded function. For $i=1,2, \dots,N$, $x_{i-1}$ belongs univocally to $I_i$ and using definition of $Tg$
$$(Tg)(x_{i-1})= F_i \big(L_i^{-1}(x_{i-1}), g\big(L_i^{-1}(x_{i-1})\big)\big)=F_i\big(x_0,g(x_0)\big)=F_i \big(x_0,y_0\big)=y_{i-1}.$$
Further, $x_N$ belongs to $I_N$ and using $N$-th piece of the definition of $Tg$ one obtains
$$(Tg)(x_N) = F_N \big(L_N^{-1}(x_N), g\big(L_N^{-1}(x_N)\big)\big)=F_N\big(x_N,g(x_N)\big)=F_N \big(x_N,y_N\big)=y_N.$$
Therefore, $Tg$ is well-defined and maps into $\mathcal{B}_{y_0,y_N} (I)$. Following the proof of previous theorem we assert that $T$ is a contraction, and consequently the Banach fixed point theorem ensures the existence of a unique fixed point $f$. It follows at once that $f(x_i)=(Tf)(x_i)=y_i$ for $i=0,1, \dots, N$, delivering the promised result.
\end{proof}
\begin{remark}
Due to the lack of ``join-up"  condition  $F_i(x_N,y_N)=F_{i+1}(x_0,y_0)$, $i=1,2,\dots, N-1$,   the fractal function $f$ in the preceding theorem has a jump  discontinuity  at each of the internal knots (and hence possibly at many other points). Same is the case, even if $q_i$, $i=1,2,\dots, N$ are assumed to be continuous on $I$. Later, to derive additional properties of bounded (discontinuous)  fractal function, we shall assume that the maps $q_i$, $i=1,2,\dots, N$ involved in the IFS are Lipschitz continuous.
\end{remark}
\textbf{Case 3: Discontinuous  fractal function}\\\\
Here we consider fractal function in $\mathcal{B}(I)$, which is not attached to a data set. Let $\{x_0,x_1,\dots,x_N\}$ be a partition of $I=[x_0,x_N]$ satisfying $x_0<x_1<\dots<x_N$. As in the previous case, let $I_i=[x_{i-1},x_i)$ for $i=1,2,\dots, N-1$ and $I_N=[x_{N-1},x_N]$. Let
$L_i: [x_0,x_N) \to [x_{i-1},x_i)$ be affinities such that $L_i(x_0)=x_{i-1}$ and $L_i(x_N-)=x_i$ for $i=1,2,\dots, N-1$ and $L_{N}: I \to I_N$ be affine map satisfying $L_N(x_0)=x_{N-1}$ and $L_N(x_N)=x_N$.  For $i=1,2,\dots, N$, let $q_i$ be bounded function. Note that we do not require any additional conditions such as join-up conditions and end point conditions for the bivariate maps $F_i(x,y)=\alpha_i y + q_i(x)$.
\begin{theorem}\label{thm3}
 The fractal function $f$  corresponding to the IFS defined via the maps $L_i$ and $F_i$  of the previous paragraph is  bounded and satisfies the self-referential equation $$f \big(L_i(x) \big) = \alpha_i f(x) + q_i(x).$$
\end{theorem}
\begin{proof}
Define a map $T: \mathcal{B}(I) \to \mathcal{B}(I)$ by
$$(Tg)(x) = F_i \big(L_i^{-1}(x), g \big(L_i^{-1}(x)\big) \big) = \alpha_i g \big( L_i^{-1}(x)\big) + q_i \big(L_i^{-1}(x)\big), ~x \in I_i,~ i=1,2,\dots,N.$$
Since $g$ and $q_i$ are in the linear space $\mathcal{B}(I)$ , it follows readily that $Tg \in \mathcal{B}(I)$ and $T$ is well-defined. As noted previously, for $g,h \in \mathcal{B}(I)$
\begin{equation*}
\begin{split}
\big|(Tg)(x)-(Th)(x)\big| =& ~ |\alpha_i| \big|g\big(L_i^{-1}(x)\big)-h\big(L_i^{-1}(x)\big) \big|\\
\le &~ |\alpha_i| \|g-h\|_\infty\\
\le &~ |\alpha|_\infty \|g-h\|_\infty,
\end{split}
\end{equation*}
and hence $$\|Tg-Th\|_\infty \le |\alpha|_\infty \|g-h\|_\infty.$$
Therefore, $T$ is a contraction and its fixed point $f$ enjoys the self-referential equation $$f \big(L_i(x) \big) = \alpha_i f(x) + q_i(x),$$ completing the proof.
\end{proof}
\begin{remark}
It is worth to note that instead of constant scaling factors, we may employ scaling functions $\alpha_i: I \to \mathbb{R}$ with suitable conditions
to provide fractal functions with more flexibility. To  to the least, we may assume that each $\alpha_i$ is bounded and $\|\alpha\|_\infty:= \max \{\|\alpha_i\|_\infty: i=1,2,\dots, N\}<1$
\end{remark}
\begin{remark}
Since fractal function $f$ appearing in Theorem \ref{thm3} is bounded, it belongs to $\mathcal{L}^\infty(I)$ and  also  to $\mathcal{L}^p(I)$ for $1 \le p < \infty$. In particular, $f$ is Lebesgue integrable.
\end{remark}
\begin{remark}\label{remalpaff}
As pointed out in  reference \cite{Barnsley1} for a  continuous function,  given  $f \in \mathcal{B}(I)$ one may choose $b \in \mathcal{B}(I)$  and consider $q_i(x):= f (L_i(x)) -\alpha_i b(x)$. The corresponding fractal function, which is termed $\alpha$-fractal function, denoted by $f^\alpha$ provides the self-referential analogue of $f$. In case $b$ depends linearly on $f$, then the correspondence $f \mapsto f^\alpha$ provides a linear operator on $\mathcal{B}(I)$, extending the notion of $\alpha$-fractal operator (see \cite{N1}) to the space $\mathcal{B}(I)$.
\end{remark}
\begin{remark}
On lines similar to Theorem \ref{thm1}, for a prescribed data set $\{(x_i,y_i): i=0,1,\dots, N\}$ by proper choices of $\alpha_i$ and $q_i$, we can construct a fractal function $f \in \mathcal{B}(I)$ (not necessarily continuous) that approximates the data in the sense that  $|f(x_i)-y_i| \le \epsilon$ for $i=0,1,\dots, N$. For instance, Lipschitz map $q_i$ can be taken as affinities $q_i(x)= q_{i0} x+ q_{i1}$, $i=1,2,\dots, N$ and the coefficients thereof can be selected so that the function values at the knots
\begin{equation*}
\begin{split}
& f(x_0) = \frac{q_{10} x_0 + q_{11}}{1-\alpha_1},\quad f(x_N)= \frac{q_{N0}x_N+ q_{N1}}{1-\alpha_N},\\
& f(x_i)= \alpha_{i+1} f(x_0) + q_{i+i,0} x_0 + q_{i+1,1},\quad i=1,2,\dots, N-1
\end{split}
\end{equation*}
are close enough to $y_i$, $i=0,1,\dots, N$.
\end{remark}
 \begin{example} We now illustrate previous results by constructing examples of  continuous and discontinuous fractal functions which interpolate or approximate the set of data  $\{(0,0), (0.5,0.5), \\(1,0)\}$. Fig. \ref{Fig1} shows the affine fractal function  in the standard setting (i.e., continuous and interpolatory) with scale vector $\alpha=(0.75, 0.75)$. Note that the graph of the FIF has Minkowski dimension $D\approx 1.585$ obtained as the unique solution of (see \cite{Barnsley1}) $$ \sum_{i=1}^N |\alpha_i| a_i^{D-1} =1.$$
 Fig.  \ref{Fig2} represents continuous affine fractal approximants, where the scale vector is taken to be $\alpha=(0.5,0.5)$ and coefficients appearing in the affinities $q_i(x)=q_{i0}x + q_{i1}$ are chosen such that $|f(x_i)-y_i| < \epsilon$ with $\epsilon= 0.1, 0.05,$ and $0.005$. Fig. \ref{Fig3} displays discontinuous fractal interpolation function corresponding to the data set wherein scaling vector is  $\alpha=(0.5,0.5)$ and the coefficient $c_1$ appearing in the affine map is chosen (at random) as $\frac{1}{8}$. In Fig. \ref{Fig4}, we break both continuity and interpolatory conditions inherent in a traditional affine FIF. The coefficients of affinities  are randomly chosen so that the graph passes close to the given data, that is,  we construct discontinuous fractal function $f$ satisfying  $|f(x_i)-y_i|< 0.1$.
\begin{figure}[h!]
\begin{center}
\begin{minipage}{0.9\textwidth}
\epsfig{file=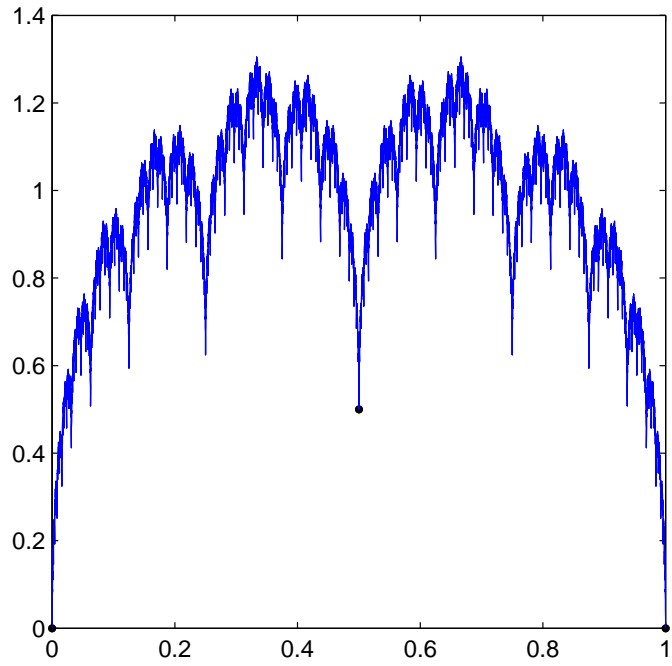,scale=0.9}
\end{minipage}\hfill
\caption{Graph of a continuous affine fractal interpolation function.} \label{Fig1}
\end{center}
\end{figure}
\begin{figure}[h!]
\begin{center}
\begin{minipage}{0.9\textwidth}
\epsfig{file=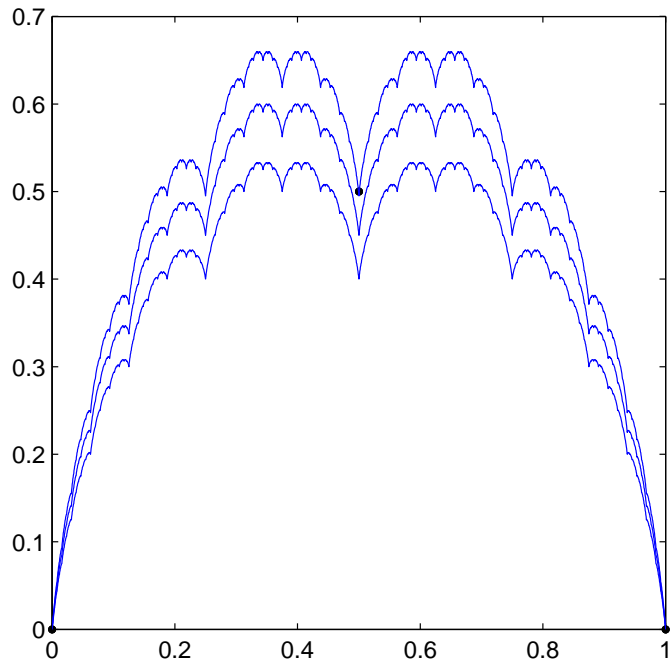,scale=0.9}
\end{minipage}\hfill
\caption{Graphs of continuous affine fractal (approximation) functions.} \label{Fig2}
\end{center}
\end{figure}
\begin{figure}[h!]
\begin{center}
\begin{minipage}{0.9\textwidth}
\epsfig{file=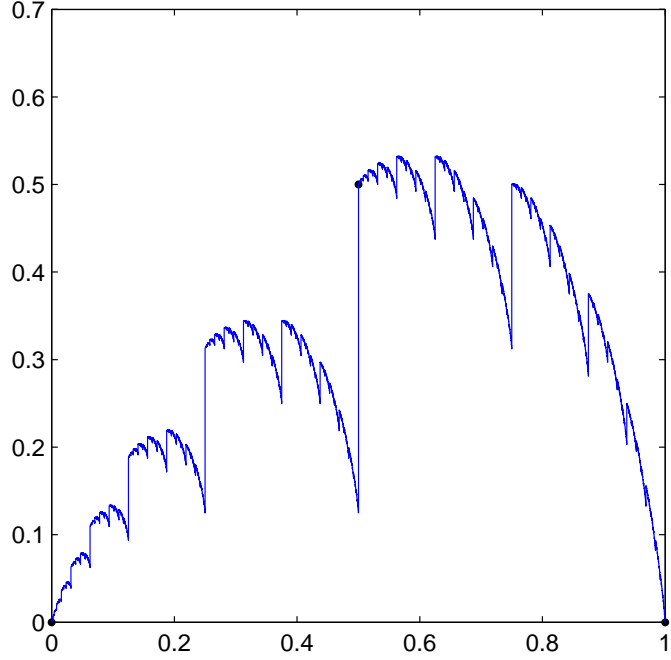,scale=0.9}
\end{minipage}\hfill
\caption{Graph of a discontinuous affine fractal interpolation function (The vertical lines display discontinuities).}\label{Fig3}
\end{center}
\end{figure}
\begin{figure}[h!]
\begin{center}
\begin{minipage}{0.9\textwidth}
\epsfig{file=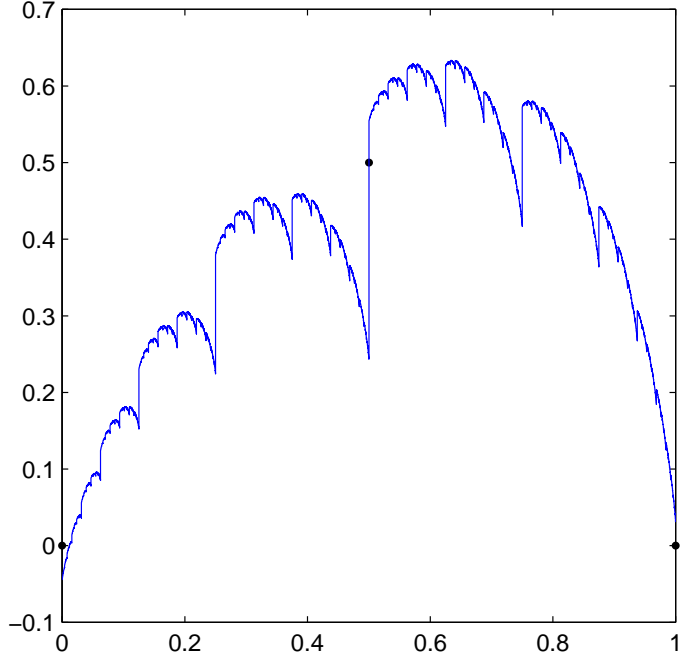,scale=0.9}
\end{minipage}\hfill
\caption{Graph of a discontinuous affine fractal (approximation) function (The vertical lines display discontinuities).} \label{Fig4}
\end{center}
\end{figure}
\end{example}
\section{Some Properties of Discontinuous Fractal Functions}\label{ffhistosec4}
Recall that a continuous fractal interpolation function can be obtained either as a fixed point of operators defined on suitable function spaces or as attractors of IFSs. In the previous section, we obtained discontinuous fractal functions as fixed points of RB-operators.
 Here we  first prove that the graph $G(f)$ of this bounded fractal function $f$ is related to the attractor of the IFS $\{I \times \mathbb{R}; W_i, i=1,2,\dots, N\}$. Further, we prove that the set of points of discontinuities of $f$ has Lebesgue measure zero. Our analysis follows closely (\cite{BHM}, Theorem 3.5, p. 403) which proves a similar result for local fractal functions. However, we note that for the validity of the theorem in \cite{BHM} suitable additional condition, for instance, Lipschitz continuity, is to be imposed on the maps $\lambda_i$ (which takes the role of $q_i$ in the present setting) that seems to be missing.\\
As a prelude, let us recall a pair of definitions.\\
If $f: X \to \mathbb{R}$ is a function on a metric space $(X,d)$, then the oscillation of $f$ on an open set $U$ is
$$ \omega_f (U) = \sup_{x \in U} f(x) - \inf_{x \in U} f(x)= \sup_{a,b \in U}|f(a)-f(b)| $$
and the oscillation of $f$ at a point $x^*$ is defined as
$$\omega_f (x^*) = \lim_{\epsilon \to 0} \omega_f \big(B_\epsilon(x^*)\big),$$
where $B_\epsilon (x^*)$ is the open ball at $x^*$ with radius $\epsilon$ defined by $B_\epsilon (x^*)= \{x \in X: d(x,x^*)<\epsilon\}.$
Note that $f$ is continuous at $x^*$ if and only if $\omega_f (x^*)=0$.
\begin{theorem}
Consider $W_i: I \times \mathbb{R} \to I \times \mathbb{R}$ defined by $W_i(x,y)=\big(L_i(x),F_i(x,y)\big)$ with $L_i$ and $F_i$ as in Theorem \ref{thm3}. Further assume that $q_i$ is Lipschitz continuous with  Lipschitz constant $Q_i$ and $Q=\max\{Q_i:i=1,2,\dots, N\}$.
Then the IFS $\{I \times \mathbb{R}; W_i, i=1,2,\dots, N\}$ is contractive with respect to a metric $d_\theta$ on $\R^2$ defined by
$$d_\theta\big((x,y),(x',y')\big)= |x-x'|+ \theta |y-y'|, $$
where $0< \theta < \dfrac{1-\max_{1\le i\le N} a _i}{Q}$. Furthermore, the unique attractor of this IFS is the closure of the graph of $f$.
\end{theorem}
\begin{proof}
Let $(x,y), (x',y') \in I \times \R$. We have
\begin{equation*}
\begin{split}
d_\theta\big(W_i(x,y),W_i(x',y')\big) =&~d_\theta \Big(\big(L_i(x), \alpha_i y+ q_i (x)\big), \big(L_i(x'), \alpha_i y'+ q_i (x')\big)\Big)\\
=&~ |L_i(x)-L_i(x')| + \theta \big| \alpha_i y+ q_i(x) -\big(\alpha_i y' + q_i (x') \big)\big|\\
\le &~ a_i |x-x'|+ \theta \big[ |\alpha_i| |y-y'| + |q_i(x)-q_i(x')| \big]\\
=&~ (a_i + \theta Q_i) |x-x'| + \theta |\alpha_i| |y-y'|\\
\le &~ \max \{a_i+ \theta Q_i, |\alpha_i|\} \big[|x-x'|+ \theta |y-y'|\big].
\end{split}
\end{equation*}
 For $\theta$ as mentioned in the statement of the theorem, it follows that $K:= \max \{a_i+ \theta Q_i, |\alpha_i|\} <1$ and consequently that $W_i$, $i=1,2,\dots, N$ are contraction maps. That is, the IFS $\{I \times \mathbb{R}; W_i, i=1,2,\dots, N\}$ is contractive, and hence it has a unique attractor, say $A$. Let $G(f):=\{(x,f(x)): x \in I\}$ denote the graph of $f$ and $\mathcal{H} (I \times \R)$ denote the space of all nonempty compact subsets of $I \times \R$ endowed with Hausdorff metric.  Consider the Hutchinson operator $W: \mathcal{H} (I \times \R) \to \mathcal{H} (I \times \R)$ defined by
\[
W(B) = \cup _{i=1}^N W_i (B)
\]
We have to show that $W (\overline{G(f)})= \overline{G(f)}$. Let us extend $W$ to $2^{I \times \R}$, the powerset of $I \times \R$. With a slight abuse of notation we shall denote the extension  also  by $W$. Consider $(x,y) \in \overline{G(f)}$. Then there exists a sequence of points $x_n \in I$  such that $\lim_{n \to \infty} (x_n, f(x_n))= (x,y)$. Since $I= \cup_{i=1}^N I_i$ and $I_i \cap I_j = \emptyset $ for all $i \neq j$, for each fixed $x_n$, we have $x_n=L_i (x_n')$, and hence
\begin{equation*}
\begin{split}
(x_n, f(x_n))= \big(L_i(x_n'), f(L_i(x_n'))\big)= &~\big(L_i(x_n'), F_i(x_n', f(x_n'))\big)\\
=&~W_i \big(x_n', f(x_n')\big)\in W(G(f)) \subseteq W(\overline{G(f)}).
\end{split}
\end{equation*}
Consequently, $(x,y) \in \overline{W(\overline{G(f)})}=W(\overline{G(f)})$, showing that
$$\overline{G(f)} \subseteq W(\overline{G(f)}).$$ Next to show that $W(\overline{G(f)})\subseteq\overline{G(f)}$, we first show that
$W(\tilde{G}) \subseteq G(f)$, where $\tilde{G}:= G(f)\setminus \big\{(x_i, f(x_i)), i=0,1,\dots, N\big\}$. Let $(x,y) \in W(\tilde{G})$. Then
$(x,y) \in W_i (\tilde{G})$ for some $i \in \{1,2,\dots, N\}$. Therefore, there is an $(x', y')$ such that $x' \in I \setminus \{x_0,x_1,\dots, x_N\}$, $y'=f(x')$, $x=L_i(x')$, and
$$y= F_i(x',y')=\alpha_i f(x')+ q_i(x')= \alpha_i f \big(L_i^{-1}(x)\big)+ q_i\big(L_i^{-1}(x)\big)=f(x),$$
so that $(x,y) \in G(f)$, proving $W(\tilde{G}) \subseteq G(f)$. Using this and continuity of the Hutchinson map
$$W(\overline{G(f)})= W(\overline{\tilde{G}}) \subseteq  \overline{W(\tilde{G})} \subseteq \overline{G(f)},$$
completing the proof.
\end{proof}
\noindent The foregoing theorem shows that the graph of a bounded  discontinuous fractal interpolant $f$ can be
approximated by the chaos game algorithm \cite{BAV}.
\begin{remark}
Note that in the case of continuous fractal function $f$, the graph $G(f)$ is closed, and hence the attractor coincides with $G(f)$.
\end{remark}
\noindent This next theorem which is just a slight variant of the Collage theorem (see \cite{BBook}) hints towards the choices of maps $W_i$ so that the bounded fractal function $f$ is close to a prescribed $\Phi \in \mathcal{B}(I)$.
\begin{theorem}
Let the mappings $W_i$, $i=1,2,\dots, N$ in the IFS used to generate the bounded discontinuous fractal function $f$ (Cf. Theorem \ref{thm3}) be chosen such that $\|\Phi- T\Phi\|_\infty < \epsilon$, where $T$ is the RB-operator whose fixed point yields $f$. Then $$\|\Phi-f\|_\infty < \frac{\epsilon}{1-|\alpha|_\infty}.$$
\end{theorem}
\begin{proof}
By the Banach fixed point theorem, the fractal function $f= \lim_{m \to \infty} T^m g$,  where $g \in \mathcal{B}(I)$ is arbitrary.  Therefore we see that
\begin{equation*}
\begin{split}
\|\Phi-f\|_\infty  = \|\Phi-\lim_{m \to \infty} T^m \Phi\|_\infty
=&~ \lim_{m \to \infty} \|\Phi-T^m \Phi\|_\infty\\
=&~\lim_{m \to \infty} \Big\| \sum_{i=1}^m (T^{i-1}\Phi- T^i \Phi)\Big\|_\infty\\
\le &~ \lim_{m \to \infty}\sum_{i=1}^m \|T^{i-1}\Phi- T^i \Phi\|_\infty\\
=&~ \lim_{m \to \infty}\sum_{i=1}^m \|T^{i-1} (\Phi-T \Phi)\|_\infty\\
\le&~ \lim_{m \to \infty}\sum_{i=1}^m |\alpha|_\infty^{i-1} \|\Phi-T \Phi\|_\infty\\
< &~ \frac{\epsilon}{1-|\alpha|_\infty},
\end{split}
\end{equation*}
providing the assertion.
\end{proof}
\begin{theorem} \label{thm4a}
 The set of points of  discontinuities for the function $f$ (Cf. Theorem \ref{thm3}) is a Lebesgue null set. In particular, $f$ is Riemann integrable.
\end{theorem}
\begin{proof}
Since $f$ is obtained by the application of Banach fixed point theorem on the Read-Bajraktarevi\'c operator $T$, $f=\lim_{n \to \infty} g_n$, where $g_n = T g_{n-1}$, and  $g_0 \in \mathcal{B}(I)$ is arbitrary. We choose $g_0 \equiv 1$, the constant function on $I=[x_0,x_N]$. Note that
$$g_n \big(L_i(x)\big)= \alpha_i g_{n-1}(x) + q_i (x), ~ x \in I, ~i=1,2,\dots, N.$$
The function $g_1$ may have finite jump discontinuities at the internal knot points $x_1, x_2, \dots, x_{N-1}$. In general, $g_n$ may have finite jump discontinuities at the internal knot points $x_1, x_2, \dots, x_{N-1}$ of the partition and possibly at the image  points $L_{i_1} \circ L_{i_2} \circ \dots L_{i_{n-1}}(x_r)$ of the internal knots. Therefore by a simple arithmetic, it follows that the $n$-th pre-fractal function $g_n$ has at most $(N-1) (N^{n-1}+1)$ points of discontinuity in $I$. Let $A_n$ denote the set of points of discontinuities of $g_n$, then $|A_n| \le (N-1) (N^{n-1}+1)$. Further, assume $A= \cup_{n=1}^\infty A_n$, so that $|A|  \le \aleph_0$, that is $A$ is countable. We prove that $f$ is continuous at $x \in I \setminus A$. Choose an $\epsilon >0$. Taking into account that $q_i$, $i=1,2,\dots, N$ are Lipschitz functions with Lipschitz constants $Q_i$ and $Q=\max\{Q_i:i=1,2,\dots,N\}$ , from the fixed point equation
$$f\big(L_i(x)\big) = \alpha_i f (x) + q_i (x)$$
it follows that
\begin{equation}\label{dffri1}
\begin{split}
\omega_f \big(L_i(I)\big) =&~ \sup_{x,y \in I} \big|f(L_i(x))-f(L_i(y))\big|\\
= &~ \sup_{x,y \in I} \big|\alpha_i (f(x)-f(y))+ q_i(x)-q_i(y)\big|\\
\le & ~|\alpha|_\infty \omega_f(I) + Q|I|,
\end{split}
\end{equation}
where $|\alpha|_\infty:= \max\{|\alpha_i|: i=1,2,\dots,N\}$ and $|I|$ is the length of $I$.

Let $\Omega=\{1,2,\dots,N\}$ and let $\Omega^\infty$ be the set of all infinite sequences $\sigma=\sigma_1 \sigma_2 \sigma_3\dots$ of elements in $\Omega$. Then $(\Omega, d_{\Omega})$,  refereed to as code space, is a compact metric space,  with the metric $d_{\Omega}$  defined by $d_{\Omega} (\sigma, \omega)=2^{-k}$, where $k$ is the least index for which $\sigma_k \neq \omega_k$. For any finite code $\sigma|_k= \sigma_1 \sigma_2 \dots \sigma_k$ from (\ref{dffri1}) we obtain
\begin{equation*}
\begin{split}
\omega_f \big(L_{\sigma|_k}(I)\big) = ~& \omega_f \big(L_{\sigma_1} \circ L_{\sigma_2} \circ \dots \circ L_{\sigma_k} (I)\big)\\
\le &~ |\alpha|_\infty w_f \big(L_{\sigma_2} \circ \dots \circ L_{\sigma_k} (I)\big) + Q a^{k-1} |I|,
\end{split}
\end{equation*}
which on recursion yields
\begin{equation}\label{dffri2}
\omega_f \big(L_{\sigma|_k}(I)\big) = |\alpha|_\infty ^k \omega_f(I) + Q |I| \frac{a^k}{\big||\alpha|_\infty -a\big|}.
\end{equation}
Consider the contractive IFS $\{I; L_i,i=1,2,\dots,N\}$ with attractor $I$. The limit $\lim_{k \to \infty} L_{\sigma_1} \circ L_{\sigma_2} \circ \dots L_{\sigma_k}(a)$ is a single point independent of $a \in I$ and the coding map
 $\pi: \Omega^ \infty \to I$
 $$\pi (\sigma):= \lim_{k \to \infty} L_{\sigma_1} \circ L_{\sigma_2} \circ \dots L_{\sigma_k}(a)$$
 is continuous and surjective. Therefore there exists $\sigma \in \Omega^ \infty$ such that
 $$x= \pi (\sigma) = \cap_{k=1}^\infty L_{\sigma|_k} (I).$$
For any $k \in N$, there exists a compact interval $I_k$ such that
$$x \in I_k \subset \cap_{i=1}^k L_{\sigma|_i} (I).$$
Set $J:= L_{\sigma|_k}^{-1} (I_k)$, where $\sigma|_k^{-1}= L_{\sigma_k} ^{-1} \circ L_{\sigma_{k-1}}^{-1} \circ \dots \circ L_{\sigma_1}^{-1}$.
In view of (\ref{dffri2}) we obtain
$$\omega_f(I_k)= \omega_f \big(L_{\sigma|_k}(J)\big) \le |\alpha|^k \omega_f(I)+ Q |J| \frac{a^k}{\big||a|-|\alpha|_\infty\big|}.$$
Since $f$ is bounded on $I$, $|J| \le x_N-x_0$, and $a^k \to 0$ as $k \to \infty$, we can choose $k$ to be large enough so that each summand in the previous inequality is less than  $\frac{\epsilon}{2}$. Consequently, $\omega_f(I_k)< \epsilon$. Since $\epsilon$ is arbitrary, we deduce that $\omega_f(x)=0$, and hence $f$ is continuous at $x \in I \setminus A$. That $f$ is Riemann integrable follows now from a standard result in analysis which states that a real-valued bounded function $f$ is Riemann integrable if and only if set of  points of discontinuities for $f$ has Lebesgue measure zero, see for instance, \cite{Rudin}.
\end{proof}
\begin{theorem}\label{thm4}
Let $f$ be the discontinuous  fractal function given in Theorem \ref{thm3}. For each $m \in \mathbb{N} \cup \{0\}$, the moment integral
$$f_m = \int_{I} x^m f(x)~ \mathrm{d}x$$
can be explicitly evaluated recursively in terms of the lower moment integrals $f_{m-1}, f_{m-2}, \dots, f_0$, the scaling factors $\alpha_i$, and the moment
$$Q_m = \int_{I} x^m Q(x)~ \mathrm{d}x,$$
 where the function $Q: I \to \mathbb{R}$ is defined as
 $$Q(x)= q_i \circ L_i^{-1}(x), ~\text{for}~ x \in I_i.$$
\end{theorem}

\begin{proof}
In view of the previous theorem it follows at once that the moment integrals are well-defined in the Riemann sense. With a series of self-explanatory steps we have
\begin{equation*}
\begin{split}
f_m =~& \sum_{i=1}^N\int_{I_i} x^m f(x)~ \mathrm{d}x\\
=~& \sum_{i=1}^N \int_{I_i} x^m \big[\alpha_i f\big(L^{-1}(x)\big)+q_i \big(L_i^{-1}(x)\big) \big]~ \mathrm{d}x\\
=~& \sum_{i=1}^N  a_i  \alpha_i\int_{I} (a_i \tilde{x}+ b_i)^m  f(\tilde{x})~ \mathrm{d}\tilde{x}+ \int_{I}x^m Q(x)~\mathrm{d}x\\
=~& \sum_{i=1}^N \sum_{k=0}^m \alpha_i a_i^{k+1}b_i^{m-k} {m \choose k} f_k + Q_m,
\end{split}
\end{equation*}
which may be recast as
$$f_m = \frac{\sum_{k=0}^{m-1}  {m \choose k} f_k \sum_{i=1}^N \alpha_i a_i^{k+1}b_i^{m-k}+Q_m}{1-\sum_{i=1}^N \alpha_i a_i^{m+1}}.$$
Also, in particular
$$f_0= \int_{I} f(x) ~ \mathrm{d}x = \frac{\int_{I} Q(x) ~\mathrm{d}x}{1- \sum_{i=1}^N a_i \alpha_i}.$$
Since $a_i =\dfrac{x_i-x_{i-1}}{x_N-x_0}$, we get $\sum_{i=1}^N a_i =1$ and  hence on account of $|\alpha_i|<1$ it follows that $\sum_{i=1}^N a_i \alpha_i<1$.
\end{proof}
We can extend the discontinuous fractal function $f$ supported on $I$ to whole $\mathbb{R}$ by defining the
extension to zero off $I$, which will also be denoted by $f$. Now the functional equations satisfied by integral transforms of these discontinuous fractal functions can be easily obtained, which may be of interest for variety of reasons, see also \cite{Barnsley1}. The general integral transform  of a ``well behaved" function $f$ is defined as
$$\hat{f} (s) = \int_{\mathbb{R}} K(x,s) f(x) ~ \mathrm{d}x,$$
where $K(x,s)$ is a suitable function, referred to as kernel of the transformation. Using the functional equation for fractal function $f$ one obtains
\begin{equation*}
\begin{split}
\hat{f} (s) =~& \int_{I} K(x,s) f(x) ~ \mathrm{d}x\\
=~& \sum_{i=1}^N \int_{I_i} K(x,s) \big[\alpha_i f\big(L_i^{-1}(x) \big) + q_i \big(L_i^{-1}(x)\big]~\mathrm{d}x\\
=~&\sum_{i=1}^N a_i \alpha_i \int_{I} K\big(L_i(x),s\big) f(x) ~\mathrm{d}x+ \hat{Q}(s),
\end{split}
\end{equation*}
where $\hat{Q}(s)$ is the integral transform of the function $Q: I \to \mathbb{R}$ defined by $Q(x)= q_i \big(L_i^{-1}(x)\big)$ for $x \in I_i$.\\
This being said, it is tempting to examine transform of a fractal function with some special  choices of kernel functions.\\
\textbf{Case 1: Laplace transform}\\  Here $K(x,s)= e^{-sx}$ for $x >0$.  Then
\begin{equation*}
\begin{split}
\hat{f} (s) =~&\sum_{i=1}^N a_i \alpha_i \int_{I} e^{-s(a_ix+b_i)} f(x) ~ \mathrm{d}x+ \hat{Q}(s)\\
=~& \sum_{i=1}^N a_i \alpha_i e^{-s b_i} \hat{f}(a_i s) + \hat{Q}(s).
\end{split}
\end{equation*}
\textbf{Case 2: Stieltjes transform}\\
Taking $K(x,s)= \frac{1}{s-x}$ we obtain
\begin{equation*}
\begin{split}
\hat{f} (s) =~&\sum_{i=1}^N a_i \alpha_i \int_{I}\frac{1}{s-(a_ix+b_i)}f(x) ~ \mathrm{d}x+ \hat{Q}(s)\\
=~& \sum_{i=1}^N \alpha_i \hat{f} \Big(\frac{1}{a_i}(s-b_i)\Big)+\hat{Q}(s).
\end{split}
\end{equation*}
\textbf{Case 3: Fourier transform}\\
For the kernel $K(x,s)=e^{jsx}$, where $j$ is the square root of $-1$, we have

\begin{equation*}
\begin{split}
\hat{f} (s) =~&\sum_{i=1}^N a_i \alpha_i \int_{I} e^{js(a_ix+b_i)}f(x) ~ \mathrm{d}x+ \hat{Q}(s)\\
=~& \sum_{i=1}^N a_i \alpha_i e^{j s b_i} \hat{f}( a_i s)+\hat{Q}(s).
\end{split}
\end{equation*}
In case of uniformly spaced knot sequence we obtain an explicit expression as follows. \\
Denoting $\Lambda(s) =\frac{1}{N}\sum_{i=1}^N \alpha_i e^{j s b_i}$,  for equidistant knots, i.e., for  $a_i=\frac{1}{N}$  previous equation yields
\begin{equation}\label{dfftfm}
\hat{f} (s) =  \Lambda(s) \hat{f}(\frac{s}{N})+ \hat{Q}(s).
\end{equation}
Applying Equation (\ref{dfftfm}) recursively, we obtain
\begin{equation}\label{dfftfm2}
\hat{f} (s) = \Big[\prod_{i=1}^k \Lambda \big(\frac{s}{N^{i-1}}\big)\Big] \hat{f}\big(\frac{s}{N^k}\big)+ \sum_{i=0}^{k-1} \Big[\prod_{m=1}^i\Lambda \big(\frac{s}{N^{m-1}}\big)\Big] \hat{Q}\big(\frac{s}{N^i}\big),
\end{equation}
where the empty product $\prod_{m=1}^0 \Lambda \big(\frac{s}{N^{m-1}}\big)=1.$ Note that the previous equation extemporizes [Equation 3.4, \cite {Mas2}, p. 177].\\
We have
$$ |\Lambda(s)| \le \frac{1}{N}\sum_{i=1}^N |\alpha_i e^{j s b_i}| \le |\alpha|_\infty,$$
and therefore
$$ \Big|\prod_{i=1}^k \Lambda \big(\frac{s}{N^{i-1}}\big)\Big| \le |\alpha|_\infty^k \to 0 ~\text{as}~ k \to 0.$$
The previous observation in conjunction with boundedness of $\hat{f}$ asserts that the first summand in (\ref{dfftfm2}) approaches zero as $k \to \infty$. Since $$ \Big|\prod_{m=1}^i\Lambda \big(\frac{s}{N^{m-1}}\big) \hat{Q}\big(\frac{s}{N^m}\big)\Big| \le M |\alpha|_\infty ^i,$$
where $M$ is such that $|\hat{Q}(s)| \le M$ for all $s\in \R$, from (\ref{dfftfm2}) we obtain
$$ \hat{f} (s) = \sum_{i=0}^{\infty} \Big[\prod_{m=1}^i\Lambda \big(\frac{s}{N^{m-1}}\big)\Big] \hat{Q}\big(\frac{s}{N^i}\big).$$
\section{Fractal Histopolation}\label{ffhistosec5}
Suppose that a sequence of  strictly increasing knots $\{x_0, x_1, \dots, x_N\}$ and a histogram $F=\{f_1,f_2,\dots,f_N \}$, where  $f_i \in \R$ is the frequency for the class $[x_{i-1}, x_i)$, $i=1,2,\dots, N$ are given. For $i=1,2,\dots, N$, let $h_i:= x_i-x_{i-1}$ represent the step size. In fractal histopolation, we match average pixel intensities with our fractal function, in contrast to matching point value as done with interpolation. That is, we seek for  an integrable fractal function $f$  satisfying ``area" matching condition
\begin{equation} \label{eqhist1}
\int_{x_{i-1}}^{x_i} f(x) ~\mathrm{d}x =h_i f_i.
\end{equation}
Consider the IFS defined by the maps
\begin{equation} \label{eqhis2}
L_i(x) = a_i x + b_i, ~~ F_i(x,y) = \alpha_i y + q_i(x), \quad i=1,2,\dots, N,
\end{equation}
where $|\alpha_i| <1$ and $q_i : I \to \mathbb{R}$ is Lipschitz continuous map. From Theorem \ref{thm3} and Theorem \ref{thm4a} it follows that the corresponding fractal function is Riemann integrable and satisfies
$$f(x) = \alpha_i f \big(L_i^{-1}(x) \big) + q_i ^{-1}\big(L_i^{-1}(x) \big), \quad x \in I_i,~~ i=1,2,\dots, N.$$
The parameters that can be varied are scaling factors $\alpha_i$ and functions $q_i$, $i=1,2,\dots, N$. The histopolation condition prescribed in Equation (\ref{eqhist1}) necessitates
\begin{equation}\label{fh1}
\begin{split}
h_i f_i =&~ \int_{I_i} f(x) ~\mathrm{d} x\\
=&~ \int_{I_i} \Big[ \alpha_i f \big(L_i^{-1}(x)\big)+ q_i \big(L_i^{-1}(x)\big) \Big]~\mathrm{d} x\\
=&~ \alpha_i a_i \int_{I} f(x)~ \mathrm{d} x + a_i \int_{I} q_i(x)~ \mathrm{d} x\\
=&~ \alpha_i a_i \sum_{i=1}^N h_i f_i +  a_i \int_{I} q_i(x)~ \mathrm{d} x, ~i=1,2,\dots, N.
\end{split}
\end{equation}
Assume $\sum_{i=1}^N h_i f_i \neq 0$. If $q_i$ are a priori fixed maps, then in the previous equation only unknown is the scaling factor $\alpha_i$, which is obtained via
\[
\alpha_i = \frac{h_i f_i -a_i \int_{I} q_i(x)~\mathrm{d} x}{a_i \sum_{i=1}^N h_i f_i}, \quad i=1,2,\dots, N.
\]
 However, this solution may not be feasible, since we  require that $|\alpha_i|<1$ for all $i=1,2,\dots, N.$ (or a less stringent condition
 $\big[\sum_{i=1}^N a_i |\alpha_i|^p \big]^{\frac{1}{p}}$ if we work in $\mathcal{L}^1(I)$ instead of $\mathcal{B}(I)$).  Thus, in principle,  the problem demands a constrained optimization. In practice, for a quicker solution, we can fix scaling factors $\alpha_i$ a priori and treat $q_i$ as unknown functions to be determined suitably so that the corresponding bounded integrable fractal function $f$ satisfies Equation (\ref{eqhist1}).
 \begin{proposition}\label{histprop1}
 Let a sequence of  strictly increasing knots $\{x_0, x_1, \dots, x_N\}$ and a histogram $F=\{f_1,f_2,\dots,f_N \}$ be given. Consider the IFS $\{I \times \mathbb{R}; W_i:i=1,2,\dots, N\}$ defined through the maps given in Equation (\ref{eqhis2}). Assume that the scaling factors are selected at random so that $|\alpha_i|<1$ for $i=1,2,\dots, N$. The corresponding fractal function solves the histopolation problem \big(Cf. Equation (\ref{eqhist1})\big) if and only if the function $q_i$ satisfies
 \begin{equation}\label{eqhis2a}
 \int_{I} q_i(x) ~\mathrm{d} x = \frac{h_if_i- \alpha_ia_i\sum_{i=1}^N h_i f_i}{a_i},  \quad i=1,2,\dots, N.
 \end{equation}
 \end{proposition}
 \begin{proof}
 Necessary condition follows at once  from Equation (\ref{fh1}). From the functional equation for $f$  one obtains (see Theorem \ref{thm4})
 $$ \int_{I} f(x)~\mathrm{d} x= \frac{\sum_{i=1}^N a_i \int_{I} q_i(x)~\mathrm{d} x}{1-\sum_{i=1}^N a_i \alpha_i}.$$
 Therefore,
 \begin{equation*}
 \begin{split}
 \int_{I_i} f(x)~\mathrm{d} x=&~ \alpha_i a_i \int_{I} f(x)~ \mathrm{d} x + a_i \int_{I} q_i(x)~ \mathrm{d} x\\=&~ \alpha_i a_i \frac{\sum_{i=1}^N a_i \int_{I} q_i(x)~\mathrm{d} x}{1-\sum_{i=1}^N a_i \alpha_i}+ a_i \int_{I} q_i(x)~ \mathrm{d} x.
 \end{split}
 \end{equation*}
 Substituting the stated condition on $q_i$ in the previous equation we can deduce that $\int_{I_i} f(x)~\mathrm{d} x=h_i f_i$, completing the proof.
 \end{proof}
In what follows, we shall outline some choices for $q_i$ satisfying condition in Proposition \ref{histprop1}. For instance, taking $q_i$ as affine maps  $q_i(x)=q_{i0} x+ q_{i1}$, $i=1,2,\dots, N$, we obtain
\begin{equation}\label{fh2}
h_i f_i = \alpha_i a_i \sum_{i=1}^N h_i f_i + a_i (x_N-x_0)\Big[ \frac{q_{i0}}{2} (x_N+x_0) + q_{i1}  \Big], \quad i=1,2,\dots, N.
\end{equation}
 Since $\alpha_i \in (-1,1)$ are chosen as parameters as in the case of affine fractal interpolation function,  the above system consists of $N$ linear equations in $2N$ unknowns $q_{i0}$ and $q_{i1}$, $i=1,2,\dots, N$. Treating $q_{i0}$ also as parameters we obtain
 \begin{equation}\label{fh2a}
q_{i1}= \frac{h_i f_i -a_i \alpha_i \sum_{i=1}^N h_i f_i}{a_i (x_N-x_0)}-\frac{q_{i0}}{2}(x_N+x_0), \quad i=1,2,\dots, N.
\end{equation}
Next suppose that we are interested to construct a continuous fractal histopolant corresponding to strictly increasing knots $\{x_0, x_1, \dots, x_N\}$ and a histogram $F=\{f_1,f_2,\dots,f_N \}$. Choose $y_0$ and $y_N$ arbitrary. Recall from Section \ref{ffhistosec1} that the fractal function  corresponding to Equation (\ref{eqhis2}) is continuous if map $q_i$, $i=1,2,\dots, N$ satisfies
\begin{equation}\label{hieq2}\left\{
\begin{split}
q_1(x_0) =&~ y_0 (1-\alpha_1)\\
q_N(x_N)=&~ y_N (1-\alpha_N)\\
\alpha_{i+1} y_0+ q_{i+1} (x_0) =&~ \alpha_i y_N + q_i (x_N), \quad i=1,2,3,\dots, N-1.
\end{split}\right.
\end{equation}
Then the next proposition follows at once.
\begin{proposition}\label{histprop2}
 Let a sequence of  strictly increasing knots $\{x_0, x_1, \dots, x_N\}$ and a histogram $F=\{f_1,f_2,\dots,f_N \}$ be given. Consider the IFS $\{I \times \mathbb{R}; W_i:i=1,2,\dots, N\}$ defined through the maps given in Equation (\ref{eqhis2}). Assume that the scaling factors are selected at random so that $|\alpha_i|<1$ for $i=1,2,\dots, N$. The corresponding fractal function is continuous and solves the histopolation problem \big(Cf. Equation (\ref{eqhist1})\big) if  the function $q_i$ satisfies system of equations governed by  Equation(\ref{eqhis2a}) and Equation(\ref{hieq2}).
 \end{proposition}
It naturally raised the question of solvability of
the system mentioned in the foregoing Proposition. To this end, we make some remarks. If we take $q_i(x)= q_{i0} x+ q_{i1}$ and $\alpha_i \in (-1,1)$ as adjustable parameters, then Proposition \ref{histprop2} provides a system of $2N+1$ linear equations with $2N+2$ unknowns $y_0$, $y_N$, $q_{i0}$, and $q_{i1}$, $i=1,2,\dots, N$, choosing
one of the unknowns, say $y_0$, arbitrarily, the resulting square system of linear equations may be solved.
Alternatively, one may proceed as follows. Assume values to $q_1(x_N), q_2(x_N), \dots, q_{N-1}(x_N)$, $y_0$, and $y_N$ so that Equation (\ref{hieq2}) specifies $q_i(x_0)$ and $q_i(x_N)$ for $i=1,2,\dots, N$. Let  $q_i (x_0)= \beta_i$, and $q_i(x_N)= \gamma_i$. Equation \ref{eqhis2a}
is equivalent to
$$\int_{I} q_i(x)~ \mathrm{d} x= (x_N-x_0) \Big[\frac{h_i f_i -\alpha_i a_i \sum_{i=1}^N h_i f_i}{a_i (x_N-x_0)}\Big],~~i=1,2,\dots, N.$$
Therefore, in this case, the problem reduces to that of solving $N$ histopolation problems with boundary conditions
\begin{equation}\left\{
\begin{split}
& \int_{I} q_i(x)~ \mathrm{d} x=(x_N-x_0) \Big[\frac{h_i f_i -\alpha_i a_i \sum_{i=1}^N h_i f_i}{a_i (x_N-x_0)}\Big]\\& q_i(x_0) =\beta_i, ~ q_i(x_N)= \gamma_i, \quad i=1,2,\dots, N,
 \end{split}\right.
 \end{equation}
for which one can employ methods of histopolation by traditional nonrecursive functions, see, for instance, \cite{FOT}. Note that the corresponding fractal histopolant $f$ is continuous, but non-differentiable in general. For a special choice of $q_i$ used in the definition of $\alpha$-fractal function (see Remark \ref{remalpaff}), construction of continuous fractal histopolant $f$ seems to be rather easy. To this end, let $g \in \mathcal{C}(I)$ and $b \in \mathcal{C}(I)$ be such that $b \not \equiv g$, $b (x_0) = g(x_0)$ and $b(x_N)=g(x_N)$. Consider the IFS $\{I \times \R; W_i(x,y)=(L_i(x), F_i(x,y)),i=1,2,\dots, N\}$, where
\[
F_i(x,y)= \alpha_i y + g \circ L_i(x)-\alpha_i b(x).
\]
Corresponding fractal function $f \in \mathcal{C}(I)$ satisfies
\[
f(x) = g(x) + \alpha_i (f-b)\big(L^{-1}(x)\big).
\]
Histopolation condition in Equation (\ref{eqhis2a}) reads as
\[
\int_{I_i} g(x) ~ \mathrm{d} x- \alpha_i \int_{I} b(x) ~ \mathrm{d} x = a_i^{-1}\Big[h_i f_i - a_i \alpha_i \sum_{i=1}^N h_if_i\Big].
\]
As mentioned earlier, one can solve $N+1$ histopolation problems $\int_{I_i} g(x) = \frac{h_if_i}{a_i}$ and $\int_{I} b(x)=\sum_{i=1}^N h_if_i$ with boundary conditions $b(x_0)=g(x_0)$ and $b(x_N)=g(x_N)$.
It is worthwhile to mention that a fractal histospline $f$ of continuity $\mathcal{C}^k$ for the knot sequence $\{x_0, x_1, \dots, x_N\}$ can be obtained by differentiating a $\mathcal{C}^{k+1}$-continuous fractal spline interpolating the data $\{(x_i,y_i):i=0,1,\dots,N\}$, where, for instance,  $y_0=0$ and $y_i=y_{i-1} + h_if_i$ for $i=1,2,\dots, N$. Fractal splines interpolating a prescribed data  have received much attention in the literature, see for example \cite{BH,CV,NS1}.\\
\begin{example}
Consider the knot sequence $\{0, 1/2, 1\}$ and a histogram $F=\{2,3\}$. We construct  area true approximants of the histogram $F$ by using integrable fractal functions. Consider the IFS defined by the maps $$L_1(x)=\frac{1}{2}x,~~ L_2(x)=\frac{1}{2}x+\frac{1}{2}, \quad F_1(x,y)=\frac{1}{2}y+x+\frac{1}{4},~~F_2(x,y)=\frac{1}{2}y+2 x+ \frac{3}{4}.$$ Here the coefficients $q_{10}$ and $q_{20}$ appearing in the affinities are taken (at random) as $1$ and $2$ respectively, and the other coefficients are calculated using Equation (\ref{fh2a}). Resulting discontinuous fractal function given by
\[
f(x)=
\begin{cases}
\quad \frac{1}{2} f(2x) + 2x+ \frac{1}{4}        \quad\qquad\text{if}\quad
x\in [0,\frac{1}{2}),\\
\\
\frac{1}{2} f(2x-1)+4x-\frac{5}{4}\quad\text{if}\quad x \in [\frac{1}{2},1],
\end{cases}
\]
satisfies histopolation conditions $\int_{0}^{\frac{1}{2}} f(x) ~ \mathrm{d} x=1$ and $\int_{\frac{1}{2}}^1 f(x) ~ \mathrm{d} x=\frac{3}{2}$, see Figure \ref{Fig5}. Assume that the problem demands a continuous fractal histopolant. Bearing  Proposition \ref{histprop2} in mind, taking scale factors $\alpha_1=\alpha_2=0.5$ and assuming the value $y_0$ of the histopolant at the end point $x_0=0$ to be $0$, we solve the linear system to obtain $q_{10}=\frac{3}{2}$, $ q_{11}=0$, $q_{20}=-\frac{3}{2}$ and $q_{21}=\frac{5}{2}$. Corresponding continuous fractal histopolant satisfying the functional equation
\[
f(x)=
\begin{cases}
\quad \frac{1}{2} f(2x) + 3x       \quad\qquad\text{if}\quad
x\in [0,\frac{1}{2}),\\
\\
\frac{1}{2} f(2x-1)-3x +4\quad\text{if}\quad x \in [\frac{1}{2},1],
\end{cases}
\]
is depicted in Figure \ref{Fig6}. The $\mathcal{C}^1$-continuous histospline in Figure \ref{Fig7} is obtained by differentiating a cubic spline fractal function $g$ interpolating the data set $\{(0,0), (\frac{1}{2},1), (1,\frac{5}{2})\}$. For details on cubic spline FIF, the reader may consult \cite{BH,CV}. Corresponding smooth  histopolant satisfies the self-referential equation

\[
f(x)=
\begin{cases}
\quad 0.4 f(2x) -1.4616x^2 +0.4x+2.0436    \quad\qquad\text{if}\quad
x\in [0,\frac{1}{2}),\\
\\
0.4 f(2x-1)-0.3384 x^2 +1.676x +0.9404\quad\text{if}\quad x \in [\frac{1}{2},1].
\end{cases}
\]

 \begin{figure}[h!]
\begin{center}
\begin{minipage}{0.9\textwidth}
\epsfig{file=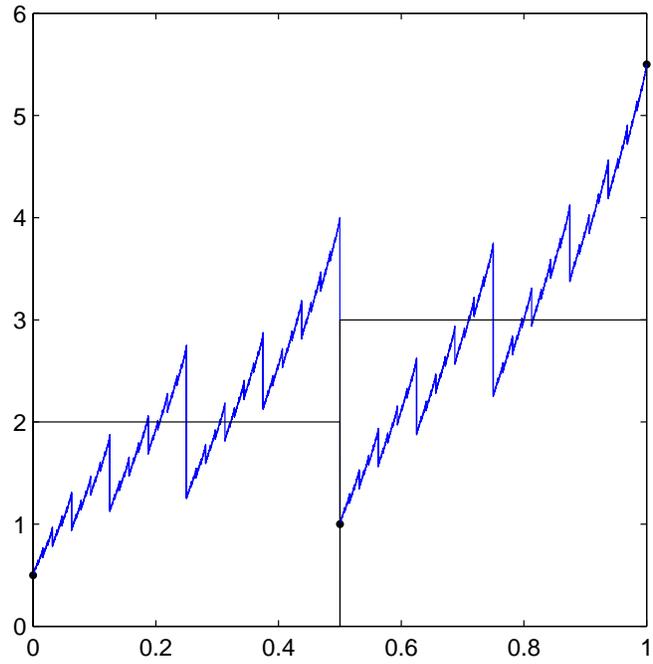,scale=0.9}
\end{minipage}\hfill
\caption{Histogram and discontinuous affine fractal histopolant.}\label{Fig5}
\end{center}
\end{figure}

\begin{figure}[h!]
\begin{center}
\begin{minipage}{0.9\textwidth}
\epsfig{file=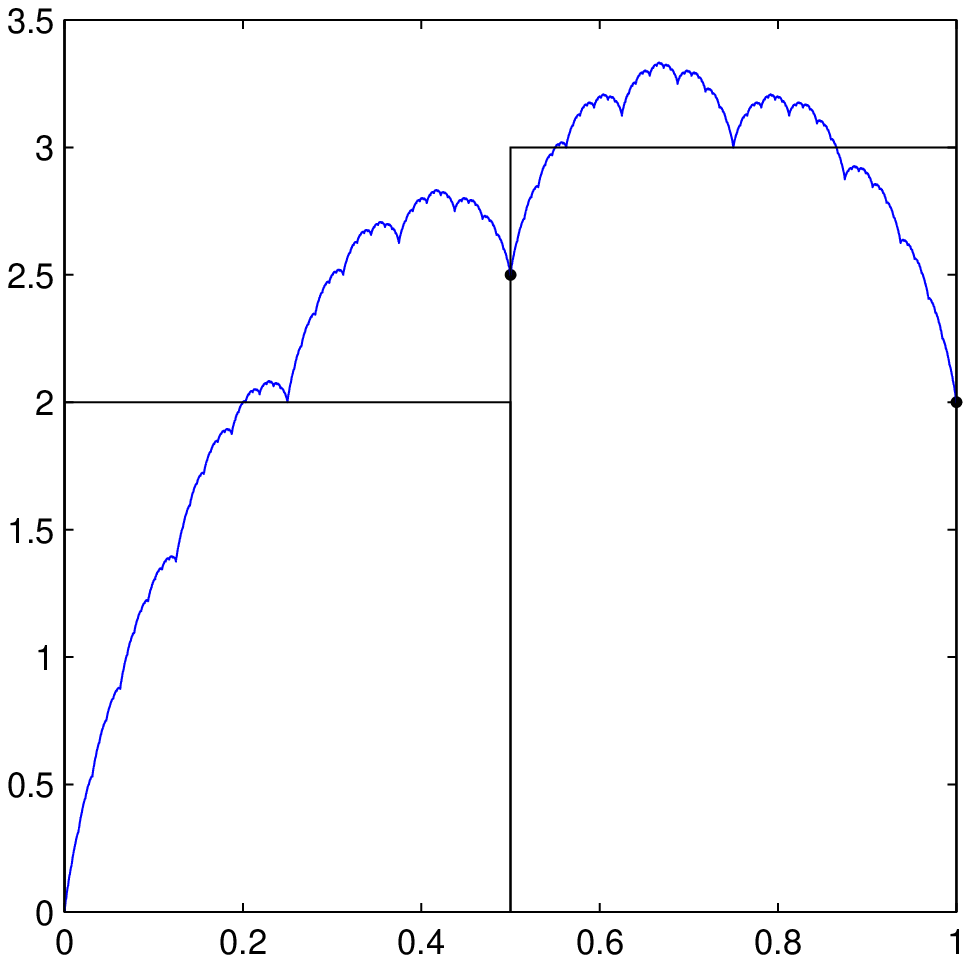,scale=0.9}
\end{minipage}\hfill
\caption{Histogram and continuous affine fractal histopolant.}\label{Fig6}
\end{center}
\end{figure}

\begin{figure}[h!]
\begin{center}
\begin{minipage}{0.9\textwidth}
\epsfig{file=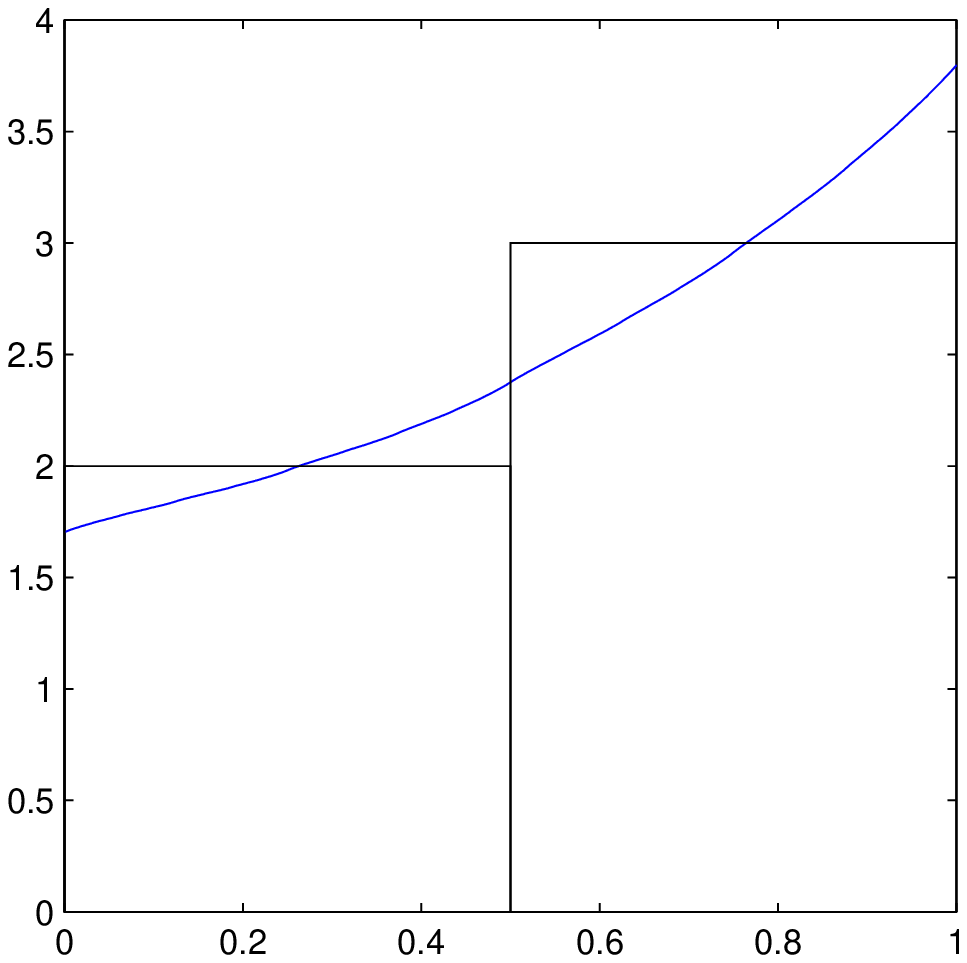,scale=0.9}
\end{minipage}\hfill
\caption{Histogram and $\mathcal{C}^1$-continuous histospline.}\label{Fig7}
\end{center}
\end{figure}
\end{example}
Fractal histopolants may be used to model planar data with prescribed Minkowski dimension which controls  the selection of scaling factors.
Minkowski and Hausdorff dimensions of a more general fractal function, for instance, bounded discontinuous fractal function, continue to remain as an open problem. There are many other strategies for identification  of free parameters in  fractal histopolation, and quite often the  particular nature of the modeling problem states the type of optimization to be employed. These also deserve future investigation.

\end{document}